\documentclass[a4paper,12pt,reqno]{amsart}
\usepackage{amsfonts}
\usepackage{amsmath}
\usepackage{amssymb}
\usepackage[a4paper]{geometry}
\usepackage{mathrsfs}
\usepackage[colorlinks]{hyperref}
\renewcommand\eqref[1]{(\ref{#1})} 
\numberwithin{equation}{section}
\theoremstyle{plain}
\newtheorem{thm}{Theorem}[section]

\newtheorem{cor}[thm]{Corollary}

\theoremstyle{definition}

\begin{document}

   \title[A comparison principle for nonlinear heat Rockland operators]
   {A comparison principle for nonlinear heat Rockland operators on graded groups}

\author[M. Ruzhansky]{Michael Ruzhansky}
\address{
  Michael Ruzhansky:
  \endgraf
  Department of Mathematics
  \endgraf
  Imperial College London
  \endgraf
  180 Queen's Gate, London SW7 2AZ
  \endgraf
  United Kingdom
  \endgraf
  {\it E-mail address} {\rm m.ruzhansky@imperial.ac.uk}
  }
 \author[D. Suragan]{Durvudkhan Suragan}
\address{
	Durvudkhan Suragan:
	\endgraf
	Department of Mathematics
	\endgraf
	School of Science and Technology, Nazarbayev University
	\endgraf
	53 Kabanbay Batyr Ave, Astana 010000
	\endgraf
	Kazakhstan
	\endgraf
	{\it E-mail address} {\rm durvudkhan.suragan@nu.edu.kz}
}

\thanks{The authors were supported in parts by the EPSRC
 grant EP/R003025/1 and by the Leverhulme Grant RPG-2017-151,
 as well as by the MESRK grant AP05130981. No new data was collected or generated during the course of research.}


     \begin{abstract}
In this note we show a comparison principle for nonlinear heat Rockland operators
on graded groups. We give a simple proof for it using purely algebraic relations. As an application of the established comparison principle we prove the global in $t$-boundedness of solutions for a class of nonlinear equations for the heat $p$-sub-Laplacian on stratified groups.
     \end{abstract}
     \maketitle

\section{Introduction}

A connected simply connected Lie group $\mathbb{G}$ is called a graded (Lie) group if its Lie algebra admits a gradation. For example, the Euclidean space, the Heisenberg group or any stratified group or a homogeneous Carnot group are examples of graded groups.

The main class of operators that we are dealing with in this paper are the so-called {\em Rockland operators}, the class of operators introduced in \cite{rockland_78} as left-invariant differential operators whose infinitesimal representations are injective on smooth vectors, and conjectured to coincide with hypoelliptic operators.
In this short note it will be much simpler to adopt another equivalent definition following Helffer and Nourrigat's resolution of the Rockland conjecture in \cite{HN-79}. Namely,  we understand by a Rockland operator {\em any left-invariant homogeneous hypoelliptic differential operator on $\mathbb{G}$.} This will allow us to avoid the language of the representation theory, and we can refer to  \cite[Section 4.1]{FR} for a detailed exposition of these equivalent notions.

The considered setting is most general since it is known that if there exists a left-invariant homogeneous hypoelliptic differential operator on a nilpotent Lie group then the group must be graded (see e.g. \cite[Proposition 4.1.3]{FR}, also for a historical discussion). Such group can be always identified with a Euclidean space $\mathbb{R}^N$ for some $N$, with a polynomial group law. We refer to the recent monograph \cite{FR} for definitions and other properties.

Note that the standard Lebesgue measure is the Haar measure for $\mathbb{G}$. 
Let $\Omega  \subset \mathbb{G}$ be a bounded set with smooth boundary $\partial \Omega $ and let $T>0$ be a real number and ${Q_T} = (0,T) \times \Omega.$ 
We denote the Sobolev space by $S^{a,p}(\Omega)=S^{a,p}_{\mathcal{R}}(\Omega)$, for $a > 0$ and $1<p<\infty$, defined by the
norm
\begin{equation}\label{norm}
\|u\|_{S^{a,p}(\Omega)}:=\left(\int_{\Omega} (|\mathcal{R}^{\frac{a}{\nu}} u(x)|^{p} + |u(x)|^{p} )dx\right)^{\frac{1}{p}} ,
\end{equation}
 where $\nu$ is the homogeneous order of the Rockland operator $\mathcal{R}$.
We also define the functional class $S_{0}^{a,p}(\Omega)$ to be the completion of $C_{0}^{\infty}(\Omega)$ in the norm \eqref{norm}.
We refer to \cite[Chapter 4]{FR} or to \cite{FR17} for a general discussion of Sobolev spaces on graded groups.

  In this note we consider the initial boundary value problem 
  \begin{equation} \label{eq:birinshi}
  	\quad\left\{{\begin{array}{*{20}{l}}
  			{{u_t} - {\mathcal{R}^{\frac{a}{\nu}}}\left( {{{\left| {{\mathcal{R}^{\frac{a}{\nu}}}u} \right|}^{p - 2}}{\mathcal{R}^{\frac{a}{\nu}}}u} \right)u =  - \gamma{{\left| u \right|}^{\beta-1} }u + \alpha {{\left| u \right|}^{q - 2}}u,\quad x \in \Omega ,\;{\kern 1pt} t > 0,}\\
  			{u \left( t, x \right) = 0,\quad \;\;\;x \in \partial \Omega ,\;\;\;{\kern 1pt} t > 0,}\\
  			{u\left( 0, x \right) = {u_0}\left( x \right),\quad x \in \Omega ,}
  	\end{array}} \right.
  \end{equation}
  where $p>1,$ $\beta >0,$ $ q\ge 1,$ $\gamma \geq0$, $\alpha \geq0$,  and the initial data is ${u_0}\left( x \right) \ge 0,  {u_0}\left( x \right)\not  \equiv 0,{u_0} \in S_{0}^{a,p}\left( \Omega  \right) \cap {L^\infty }\left( \Omega  \right).$
  
Let 
  \begin{equation}
  	V: = \left\{ {v  \in {L^p}\left( {0,T;S_{0}^{a,p}\left( \Omega  \right)} \right)\left| {{\partial _t}v  \in {L^{p'}}\left( {0,T;{S^{a,p'}}\left( \Omega  \right)} \right)} \right.} \right\},\; \frac{1}{p'}+\frac{1}{p}=1.
  \end{equation}
  
If a function $u \in V \cap C\left( 0,T;{L^2}\left( \Omega  \right) \right)$, with $ u(x) = {u_0}(x)$ for a.e. $x \in \Omega, $ satisfies
  
  \begin{equation}
  	\qquad \iint_{Q_T}{\left( {{\partial _t}u\varphi  + {{\left| {{\mathcal{R}^{\frac{a}{\nu}}}u} \right|}^{p - 2}}{\mathcal{R}^{\frac{a}{\nu}}}u  {\mathcal{R}^{\frac{a}{\nu}}}\varphi  } \right)dxdt}=-\iint_{Q_T}{\left( {{ \gamma{\left| u \right|}^{\beta  - 1}}u - \alpha {{\left| u \right|}^{q - 2}}} \right)\varphi dxdt}
  \end{equation}
  for every nonnegative test-function   $\varphi \in V \cap C\left( {0,T;{L^2}\left( \Omega  \right)} \right)$,
then the function $u$ is called a {\em weak solution} of   \eqref{eq:birinshi}.

If a function $u \in V \cap C\left( 0,T;{L^2}\left( \Omega  \right) \right)$, with $ u(x) \leq {u_0}(x)$ for a.e. $x \in \Omega, $ satisfies
 
 \begin{equation}
 \qquad \iint_{Q_T}{\left( {{\partial _t}u\varphi  + {{\left| {{\mathcal{R}^{\frac{a}{\nu}}}u} \right|}^{p - 2}}{\mathcal{R}^{\frac{a}{\nu}}}u  {\mathcal{R}^{\frac{a}{\nu}}}\varphi  } \right)dxdt} \leq-\iint_{Q_T}{\left( {{\gamma{\left| u \right|}^{\beta  - 1}}u - \alpha {{\left| u \right|}^{q - 2}}} \right)\varphi dxdt}
 \end{equation}
 for every nonnegative test-function   $\varphi \in V \cap C\left( {0,T;{L^2}\left( \Omega  \right)} \right)$,
 then the function $u$ is called a {\em weak sub-solution} of   \eqref{eq:birinshi}.
 
If a function $u \in V \cap C\left( 0,T;{L^2}\left( \Omega  \right) \right)$, with $ u(x) \geq {u_0}(x)$ for a.e. $x \in \Omega, $ satisfies

\begin{equation}
\qquad \iint_{Q_T}{\left( {{\partial _t}u\varphi  + {{\left| {{\mathcal{R}^{\frac{a}{\nu}}}u} \right|}^{p - 2}}{\mathcal{R}^{\frac{a}{\nu}}}u  {\mathcal{R}^{\frac{a}{\nu}}}\varphi  } \right)dxdt} \geq-\iint_{Q_T}{\left( {{\gamma{\left| u \right|}^{\beta  - 1}}u - \alpha {{\left| u \right|}^{q - 2}}} \right)\varphi dxdt}
\end{equation}
for every nonnegative test-function   $\varphi \in V \cap C\left( {0,T;{L^2}\left( \Omega  \right)} \right)$,
then the function $u$ is called a {\em weak sup-solution} of   \eqref{eq:birinshi}.

The goal of this note is to give a simple proof of a comparison principle for the initial boundary value problem for nonlinear heat Rockland operators on graded groups using pure algebraic relations, inspired by the recent work \cite{LZZ}. For thorough analysis of sub-harmonic analysis in related settings see e.g. \cite{BL13}, see also e.g. \cite{BK16} and \cite{BB17} for some related analysis and references therein. We also note that the global in time well-posedness of nonlinear wave equations for Rockland operators have been also recently investigated in \cite{Ruzhansky-Tokmagambetov:JDE}.

In Section \ref{SEC:2} we present the main result of this note and give its short proof. An application to nonlinear heat equations for the $p$-sub-Laplacian is then discussed in Section \ref{Sec3}.

\section{A comparison principle on graded groups}
\label{SEC:2}

The following result is a comparison principle for solutions of \eqref{eq:birinshi}.

\begin{thm}
	\label{main} Let $u$ and $v$ be weak sub-solution and sup-solution of \eqref{eq:birinshi}, respectively. If $u$ and $v$ are locally bounded, then $u\leq v$ a.e. in $Q_T$.
	\end{thm}

\begin{proof}
We have the relations (see, e.g. \cite{Lind})
\begin{equation*}
\left\{ \begin{array}{l}
(\left| c \right|^{p-2}c - \left|d\right|^{p-2}d)(c-d)\geq \frac{4}{p^{2}}||d|^{\frac{p-2}{2}}d-|c|^{\frac{p-2}{2}}c|^{2} \geq 0\; \text{if} \;  p\geq 2,\\
(\left| c \right|^{p-2}c - {\left| d  \right|^{d- 2}}d)(c-d) 
= (p-1)|d-c|^{2}\left(1+|c|^{2}+|d|^{2}\right)^{\frac{p-2}{2}} \geq 0\; \text{if} \; 2\geq p \geq 1. 
\end{array} \right.
\end{equation*}
Setting $c=\mathcal{R}^{\frac{a}{\nu}}u$ and $d=\mathcal{R}^{\frac{a}{\nu}}v$ we obtain 
\begin{equation}\label{A}
\iint_{Q_T}\left( {{{\left| {{\mathcal{R}^{\frac{a}{\nu}}}u} \right|}^{p - 2}}{\mathcal{R}^{\frac{a}{\nu}}}u - {{\left| {{\mathcal{R}^{\frac{a}{\nu}}}v} \right|}^{p - 2}}{\mathcal{R}^{\frac{a}{\nu}}}v } \right)\left( {{\mathcal{R}^{\frac{a}{\nu}}}u - {\mathcal{R}^{\frac{a}{\nu}}}v } \right)dxdt\geq 0.
\end{equation}
On the other hand, by using the straightforward inequalities 
\begin{equation*}
\left\{ \begin{array}{l}
{\left| u \right|^{\beta  - 1}}u - {\left| v  \right|^{\beta - 1}}v  = u^{\beta} - v ^{\beta} > 0,\quad \text{for} \quad  u> v  > 0,\\
\left| u \right|^{\beta-1}u - {\left| v  \right|^{\beta- 1}}v 
 = u^{\beta} + \left|v\right|^{\beta} > 0,\quad \text{for} \quad u > 0 > v, \\
{\left| u \right|^{\beta-1}}u - {\left|v\right|^{\beta- 1}}v  =  - \left| u \right|^{\beta} + \left|v\right|^{\beta } > 0,\quad \text{for} \quad 0 > u > v,
\end{array} \right.
\end{equation*}
we get that 
\begin{equation}
 \iint_{Q_T} \left( {{{\left| u \right|}^{\beta  - 1}}u - {{\left| v  \right|}^{\beta  - 1}}v } \right)\varphi \ dxdt \geq 0,
\end{equation}

where $\varphi  := \max\left\{ {u - v ,0} \right\}$, thus $\varphi \left( {0,x} \right) = 0$ and $\varphi \left( {t,x} \right)\left| {_{x \in \partial \Omega }} \right. = 0.$ Therefore, by the definitions of sub- and sup-solutions, we obtain 

$$  \frac{1}{2}\int_{\Omega} \varphi^2(t,x) dx=\iint_{Q_t}\frac{1}{2}\partial _\tau(\varphi^2(\tau,x)) dxd\tau\leq \iint_{Q_t}\varphi{\partial _\tau}\varphi dxd\tau 
$$$$+ \iint_{Q_t}\left( {{{\left| {{\mathcal{R}^{\frac{a}{\nu}}}u} \right|}^{p - 2}}{\mathcal{R}^{\frac{a}{\nu}}}u - {{\left| {{\mathcal{R}^{\frac{a}{\nu}}}v} \right|}^{p - 2}}{\mathcal{R}^{\frac{a}{\nu}}}v } \right)\left( {{\mathcal{R}^{\frac{a}{\nu}}}u - {\mathcal{R}^{\frac{a}{\nu}}}v } \right)dxd\tau$$

$$=\iint_{Q_t}\varphi{\partial _\tau}(u-v)\ dxd\tau 
+ \iint_{Q_t}\left( {{{\left| {{\mathcal{R}^{\frac{a}{\nu}}}u} \right|}^{p - 2}}{\mathcal{R}^{\frac{a}{\nu}}}u - {{\left| {{\mathcal{R}^{\frac{a}{\nu}}}v} \right|}^{p - 2}}{\mathcal{R}^{\frac{a}{\nu}}}v } \right) {{\mathcal{R}^{\frac{a}{\nu}}}\varphi } dxd\tau$$

$$=\iint_{Q_T}{\left( {{\partial _t}u\varphi  + {{\left| {{\mathcal{R}^{\frac{a}{\nu}}}u} \right|}^{p - 2}}{\mathcal{R}^{\frac{a}{\nu}}}u  {\mathcal{R}^{\frac{a}{\nu}}}\varphi  } \right)dxdt}$$
$$-\iint_{Q_T}{\left( {{\partial _t}v\varphi  + {{\left| {{\mathcal{R}^{\frac{a}{\nu}}}v} \right|}^{p - 2}}{\mathcal{R}^{\frac{a}{\nu}}}v  {\mathcal{R}^{\frac{a}{\nu}}}\varphi  } \right)dxdt}$$
$$ \leq - \gamma\iint_{Q_t} \left( {{{\left| u \right|}^{\beta  - 1}}u - {{\left| v  \right|}^{\beta  - 1}}v } \right)\varphi \ dxdt + \alpha \iint_{Q_t}\left( {{{\left| u \right|}^{q - 2}}u - {{\left| v  \right|}^{q - 2}}v } \right)\varphi \ dxd\tau $$ 
 
$$
\leq -\gamma \iint_{Q_t} \left( {{{\left| u \right|}^{\beta  - 1}}u - {{\left| v  \right|}^{\beta  - 1}}v } \right)\varphi \ dxdt + \alpha\ \underset{u>v }{\sup} \left(\frac{ {{{\left| u \right|}^{q - 2}}u - {{\left| v  \right|}^{q - 2}}v } }{u-v}\right) \iint_{Q_t}\varphi^2 dxd\tau 
$$
\begin{equation}
\leq \alpha \ \underset{u>v }{\sup} \left(\frac{ {{{\left| u \right|}^{q - 2}}u - {{\left| v  \right|}^{q - 2}}v } }{u-v}\right) \iint_{Q_t}\varphi^2 dxd\tau,
\end{equation}
that is,
\begin{equation}
\int_{\Omega} \varphi^2(x,t) dx \leqslant L\iint_{Q_t} \varphi^2(x,\tau) dxd\tau,\quad \forall t\in[0,T),
\end{equation}
where $L=2\alpha \ \underset{u>v }{\sup} \left(\frac{ {{{\left| u \right|}^{q - 2}}u - {{\left| v  \right|}^{q - 2}}v } }{u-v}\right).$
Let us recall the Gronwall inequality for the convenience of the readers:
Let $g$ and $f$ be real-valued continuous functions defined on $[0,T)$. If 
$f'(t)\leq g(t)f(t)$ for all $t\in [0,T)$, then
$$f(t)\leq f(0)\exp \int_{0}^{t}g(s)ds,\; \forall t\in [0,T).$$ 

With $f(t)=\iint_{Q_t} \varphi^2(\tau,x) dxd\tau$, from the Gronwall inequality we have $\int_{\Omega} \varphi^2 dx = 0.$ This means that $\varphi=0$ a.e. $x \in \Omega ,$ that is, $u \le v $ a.e. $\left( {t,x} \right) \in {Q_T}.$
\end{proof}

\section{An application to $p$-sub-Laplacian equations}
\label{Sec3}

Let us demonstrate an application of Theorem \ref{main} to stratified Lie groups which is one of important classes of graded groups analysed extensively by Folland \cite{F75}. 
A stratified Lie group can be defined in many different equivalent ways (see, e.g. \cite {BLU07} or \cite{FR} for the Lie group and Lie algebra points of view, respectively). 
A Lie group $\mathbb{G}=(\mathbb{R}^{N},\circ)$ is called a stratified (Lie) group  if the following two conditions are satisfied:

(a) For each $\lambda>0$ the dilation $\delta_{\lambda}: \mathbb{R}^{N}\rightarrow \mathbb{R}^{N}$
given by
$$\delta_{\lambda}(x)\equiv\delta_{\lambda}(x^{(1)},...,x^{(r)}):=(\lambda x^{(1)},...,\lambda^{r}x^{(r)})$$
is an automorphism of the group $\mathbb{G},$ where $x^{(k)}\in \mathbb{R}^{N_{k}}$ for $k=1,...,r$
with $N_{1}+...+N_{r}=N$ and
 $\mathbb{R}^{N}=\mathbb{R}^{N_{1}}\times...\times\mathbb{R}^{N_{r}}.$ 

(b) Let $X_{1},...,X_{N_{1}}$ be the left invariant vector fields on $\mathbb{G}$ such that
$X_{j}(0)=\frac{\partial}{\partial x_{j}}|_{0}$ for $j=1,...,N_{1}.$ 
Then the H{\"o}rmander condition
$${\rm rank}({\rm Lie}\{X_{1},...,X_{N_{1}}\})=N$$
holds for every $x\in\mathbb{R}^{N},$ that is, $X_{1},...,X_{N_{1}}$  with their iterated commutators
 span the whole Lie algebra of $\mathbb{G}.$

The left invariant vector field $X_{j}$ has an explicit form:
\begin{equation}\label{Xk0}
X_{j}=\frac{\partial}{\partial x'_{j}}+
\sum_{l=2}^{r}\sum_{m=1}^{N_{l}}a_{j,m}^{(l)}(x',...,x^{(l-1)})
\frac{\partial}{\partial x_{m}^{(l)}},
\end{equation}
see also \cite[Section 3.1.5]{FR} for a general presentation.
We will also use the following notations
$$\nabla_{H}:=(X_{1},\ldots, X_{N_{1}})$$
for the horizontal gradient,

\begin{equation}\label{pLap}
\mathcal{L}_{p}f:=\nabla_{H}(|\nabla_{H}f|^{p-2}\nabla_{H}f),\quad 1<p<\infty,
\end{equation}
for the (horizontal) subelliptic $p$-Laplacian or, in short, $p$-sub-Laplacian, and

$$|x'|=\sqrt{x'^{2}_{1}+\ldots+x'^{2}_{N_{1}}}$$ for the Euclidean norm on $\mathbb{R}^{N_{1}}.$

The explicit representation \eqref{Xk0} allows us to have, for instance, the identity
\begin{equation}\label{gradgamma}
|\nabla_{H}|x'|^{\gamma}|=\gamma|x'|^{\gamma-1},
\end{equation}
and
\begin{equation}\label{divgamma}
\nabla_{H}\left(\frac{x'}{|x'|^{\gamma}}\right)=\frac{N_{1}-\gamma}{|x'|^{\gamma}}
\end{equation}
for all $\gamma\in\mathbb{R}$, $x'\in \mathbb{R}^{N_{1}}$  and 
$|x'|\neq 0$.
Here and in the sequel $x'$ means `horizontal' part of $x$, that is, $x=(x',x'')\in \mathbb{R}^{N}$, where $x'$ is an $N_{1}$-dimensional vector and $x$ is an $N$-dimensional vector.
The potential theory on stratified groups from the point of view of layer potentials was developed in \cite{Ruzhansky-Suragan:Layers}.

Let $\Omega\subset \mathbb{G}$ be a bounded open set. 

Let us consider the following functional 
$$J_{p}(u):=\left(\int_{\Omega}|\nabla_{H} u|^{p}+| u|^{p}dx\right)^{\frac{1}{p}},$$
and we define the functional class $S^{1,p}_{0}(\Omega)$ to be the completion of $C^{\infty}_{0}(\Omega)$ in the norm generated by $J_{p}$ (cf. \cite{Ruzhansky-Suragan:arxiv}).  

Let us consider the following initial boundary value problem for the $p$-sub-Laplacian, $1<p<\infty$,

 \begin{equation} \label{eq:Lp}
\qquad\left\{{\begin{array}{*{20}{l}}
	{{u_t} - \mathcal{L}_{p}u =  - {{\left| u \right|}^{\beta-1} }u +  {{\left| u \right|}^{q - 2}}u,\quad x \in \Omega ,\;\;{\kern 1pt} t > 0,}\\
	{u(t,x) = 0,\quad x \in \partial \Omega ,\;\;{\kern 1pt} t > 0,}\\
	{u( 0, x ) = {u_0}\left( x \right),\quad \;\;x \in \Omega ,}
	\end{array}} \right.
\end{equation}
where $\beta >0,$ $ q\ge 1$  and the initial data is ${u_0}\left( x \right) \ge 0,  {u_0}\left( x \right)\not  \equiv 0,{u_0} \in S_{0}^{1,p}\left( \Omega  \right) \cap {L^\infty }\left( \Omega  \right).$
  
\begin{cor}\label{thm3.2} 
Let $\Omega\subset \mathbb{G}$ be a bounded open set in a stratified group.
	Let $p\leq q<\beta+1$.
	Then a weak solution of \eqref{eq:Lp} is globally in $t$-bounded.
\end{cor} 
\begin{proof}	
 Denote $r':=\underset{x=(x',x'')\in \Omega}{\max} |x'|$, then $r' < \infty$ since $\Omega$ is bounded.  For any $x=(x',x'')\in \Omega$, let $x_0=(x'_0,x''_0)\in 	\mathbb{G} \backslash \Omega$ be such that 
	 $\varepsilon \leq |x'_0 -x'| < r'+1$. It is clear that we can take $\varepsilon\in (0,1)$.
	Let 
	\begin{equation}\label{3.6}
	V(t,x) := L e^{\sigma r},\; r = |x'-x'_0|, \; x=(x',x'') \in \Omega,
	\end{equation}
	for some positive $L$ and $\sigma$ to be chosen later.
	Define $$\mathcal{M}_p v:=v_t -\mathcal{L}_{p} v - v^{q-1}+v^{\beta}.$$ 
	By using \eqref{gradgamma} and \eqref{divgamma} let us first calculate 
	
	$$\mathcal{L}_{p}V=
	\nabla_{H}(|\nabla_{H}L e^{\sigma r}|^{p-2}\nabla_{H}L e^{\sigma r})=\nabla_{H}\left(L^{p-2}\sigma^{p-2} e^{\sigma |x'-x'_{0}| (p-2)} L \sigma e^{\sigma |x'-x'_{0}| }\frac{x'-x'_{0}}{|x'-x'_{0}|}\right)$$
	$$=\nabla_{H}\left(L^{p-1}\sigma^{p-1} e^{\sigma |x'-x'_{0}| (p-1)}  \frac{x'-x'_{0}}{|x'-x'_{0}|}\right)
	=L^{p-1}\sigma^{p-1}\sigma(p-1) e^{\sigma |x'-x'_{0}| (p-1)} $$$$+ L^{p-1}\sigma^{p-1}e^{\sigma |x'-x'_{0}| (p-1)}\frac{N_{1}-1}{|x'-x'_{0}|}=(p-1)\sigma^{p} L^{p-1}
	e^{(p-1)\sigma r}+\frac{N_{1}-1}{r}\sigma^{p-1}L^{p-1}e^{\sigma r (p-1)}.$$

	Thus, $V(t,x)$ satisfies 
	\begin{equation*}
	\mathcal{M}_p V = - (p-1)\sigma^{p}L^{p-1} e^{(p-1)\sigma r}- \frac{N_{1}-1}{r} \sigma^{p-1} L^{p-1} e^{(p-1)\sigma r} - L^{q-1}e^{(q-1)\sigma r}+L^{\beta}e^{\beta \sigma r}.
	\end{equation*}
Let us find $\sigma$ and $L$ such that $\mathcal{M}_pV \geq 0$, that is, 
		\begin{equation*}
		(p-1)\sigma^{p} L^{p-1} e^{(p-1)\sigma r}+ \frac{N_{1}-1}{r} \sigma^{p-1}L^{p-1} e^{(p-1)\sigma r}+ L^{q-1}e^{(q-1)\sigma r}  \leq L^{\beta}e^{\beta \sigma r}.
		\end{equation*}
	Multiplying both sides of the inequality by $L^{-p+1} e^{-(p-1)\sigma r}$, we have 
		\begin{equation}\label{3.9}
	(p-1)\sigma^p + \frac{N_{1}-1}{r}\sigma^{p-1} + L^{q-p} e^{(q-p)\sigma r} \leq L^{\beta+1-p}e^{(\beta+1-p)\sigma r}.
	\end{equation}
Since $\varepsilon \leq r< r'+1$, to prove \eqref{3.9} it is sufficient to have 	
	\begin{equation*}
	(p-1)\sigma^p + \frac{N_{1}-1}{\varepsilon}\sigma^{p-1} + L^{q-p} e^{(q-p)\sigma (r'+1)} \leq L^{\beta+1-p}.
	\end{equation*}
Let us choose 
	\begin{equation}\label{3.10}
	\sigma = \frac{1}{(q-p)(r'+1)},
	\end{equation}
	\begin{equation*}
	L = \max \left\{ (2e)^{\frac{1}{\beta+1-q}},\left(2\left((p-1)\sigma^p +\frac{N_{1}-1}{\varepsilon}\sigma^{p-1}\right)\right)^{\frac{1}{\beta+1-p}}\right\}
	\end{equation*}
		if $q>p$, and 

	\begin{equation}\label{3.11}
	\sigma=1, \quad L =\max \left\{ 2^{\frac{1}{\beta+1-q}},\left(2\left(p-1+\frac{N_{1}-1}{\varepsilon}\right)\right)^{\frac{1}{\beta+1-p}} \right\}
	\end{equation}
	if $q=p$.
	Then we have $\mathcal{M}_pV \geq 0$. We can assume that $L \geq \| u_0\|_{L^{\infty}(\Omega)}$, otherwise we can always multiply $L$ by a sufficiently large number. Then $V(0,x)= L e^{\sigma r}\geq u_0(x)$, that is, $V(t,x)$ is a sup-solution of \eqref{eq:Lp}. Therefore, according to Theorem \ref{main}, we have
	\begin{equation}\label{3.12}
	u(t,x) \leq Le^{\sigma(r'+1)}<\infty,\quad r'=\underset{x=(x',x'')\in \Omega}{\max} |x'|.
	\end{equation}
	The right hand side of \eqref{3.12} is independent of $t$, so that  $u(t,x)$ is globally in $t$-bounded.
\end{proof}

\end{document}